\documentclass[11pt,a4paper]{article}

\usepackage[utf8]{inputenc}
\usepackage[T1]{fontenc}
\usepackage{amsmath, amssymb, amsthm}
\usepackage{graphicx}
\usepackage{geometry}
\usepackage{cite}
\usepackage{hyperref}
\usepackage{enumitem}
\usepackage{tikz}
\usepackage{float}
\usepackage{xcolor}

\theoremstyle{plain}
\newtheorem{theorem}{Theorem}[section]

\newtheorem{corollary}[theorem]{Corollary}
\newtheorem{proposition}[theorem]{Proposition}

\theoremstyle{definition}
\newtheorem{definition}[theorem]{Definition}

\theoremstyle{remark}
\newtheorem{remark}[theorem]{Remark}

\title{Emergent Higher-Order Structure from Fast Adaptive Networks}
\author{
Christian Kuehn$^{1,2,3}$
\and
Fergal Murphy$^{1}$\thanks{Email: \texttt{fergal.murphy@tum.de}}\\[4pt]
\small $^{1}$Technical University of Munich, Department of Mathematics, Munich, Germany\\
\small $^{2}$Munich Data Science Institute (MDSI), Munich, Germany\\
\small $^{3}$Complexity Science Hub Vienna (CSH), Vienna, Austria
}
\date{\today}

\begin{document}

\maketitle

\begin{abstract}
We study adaptive network models in which coupling weights evolve on a fast time scale relative
to the phase dynamics of the nodes. Using Geometric Singular Perturbation Theory (GSPT), we prove
that, although the microscopic system is strictly pairwise, the effective slow dynamics on the
invariant slow manifold can exhibit genuinely higher-order structure. More precisely, Fenichel
reduction produces explicit $O(\varepsilon)$ triplet terms in the reduced phase dynamics. In addition, we
give a rigorous criterion ensuring that these terms are irreducible, in the sense that the reduced
vector field does not admit a pairwise decomposition in node coordinates. We derive the first-order
slow-manifold correction explicitly, formulate the irreducibility criterion via mixed second derivatives,
and verify it for the adaptive Kuramoto phase oscillator model. The results show that the class of pairwise-coupled fast--slow adaptive network systems is not closed under slow-manifold reduction.
\end{abstract}

\section{Introduction}

Adaptive network dynamical systems form a central modelling paradigm in nonlinear science~\cite{Berneretal,GrossBlasius2008}:
the interaction structure coevolves with the node states, generating feedback between
dynamics and topology. One prominent benchmark example for adaptive network dynamics includes plasticity in oscillator networks \cite{AokiAoyagi2009,Haetal,BernerSchoellYanchuk,MartensBick,KuehnYoon}. In many such models the
microscopic coupling is strictly dyadic: each node interacts with others through pairwise
terms modulated by time-dependent weights. For phase oscillator networks the classical Kuramoto model is the canonical reference point \cite{Kuramoto1984,Strogatz2000} upon which other models are built.

A common and analytically tractable regime is fast adaptation, where coupling weights
relax on a time scale much shorter than that of the node dynamics. This places adaptive
networks in the class of singularly perturbed systems and motivates Geometric Singular
Perturbation Theory (GSPT)~\cite{Fenichel1979,Kuehn2015}. In this regime,
the long-time behaviour is often organised by invariant slow manifolds on which the fast
variables are slaved to the slow node states. Expanding this slow manifold in the small
parameter $\varepsilon$ yields a systematic hierarchy of corrections to the leading-order
reduced flow on the critical manifold, which is the singular limit of the slow manifold.

In parallel, there has been substantial interest in higher-order interactions, in which the
evolution of a node depends simultaneously on more than two states. Such effects are often
represented by hypergraphs or simplicial complexes and are known to modify collective
behaviour beyond what is captured by purely pairwise models \cite{Battiston2020PhysRep,SkardalArenas1,Millan2020PRL}.
For phase oscillators, nonpairwise coupling can enlarge the admissible dynamical repertoire even in symmetric settings \cite{BickAshwinRodrigues2016Chaos}.

These developments raise a basic structural question. Higher-order structure is typically postulated at the modelling level, yet many physical and biological systems are microscopically pairwise. More rigorous mathematical derivations of higher-order structure are classically based upon phase reduction~\cite{LeonPazo,BickBoehleKuehn1} or localised near bifurcations~\cite{AshwinRodrigues,Nijholtetal}, which may not extend globally in phase space. Moreover, multi-index expressions could potentially be created or removed by general coordinate changes, so any notion of higher-order interaction should be intrinsic to the vector field and compatible with node semantics. We therefore adopt a representability-based notion: a vector field is pairwise if each component can be written as a sum of two-body terms, and genuinely higher-order otherwise. Crucially, this notion is preserved by node-respecting changes of variables and admits a simple differential certificate via mixed second derivatives.

The main question of the paper is therefore whether genuinely higher-order interaction terms can emerge intrinsically and globally from systems that are microscopically pairwise. We answer this in the affirmative for a broad class of adaptive phase oscillator networks in the fast-adaptation regime. Although the original equations are strictly dyadic, the reduced vector field on the slow manifold need not admit a pairwise decomposition in the original node coordinates. In this precise sense, higher-order structure emerges as a consequence of time-scale separation. Equivalently, the class of pairwise vector fields is not closed under slow-manifold reduction.

The mechanism is straightforward. The fast subsystem defines a normally hyperbolic critical manifold of coupling weights; Fenichel theory~\cite{Fenichel1979} provides a nearby invariant slow manifold that is a graph over the phase variables. Substituting this graph into the phase dynamics produces $O(\varepsilon)$ correction terms that depend on three phases simultaneously. The crucial point is then not merely the presence of a double sum, but whether the resulting reduced vector field is still pairwise-representable. We show that a non-zero mixed derivative with respect to two other nodes rules this out, and we verify this condition explicitly for the adaptive Kuramoto model.

The contributions of the paper are as follows.
\begin{enumerate}[label=(\arabic*)]
    \item We formulate pairwise and genuinely higher-order vector fields in a simple way that is intrinsic to node coordinates and preserved by node-respecting changes of variables.
    \item We derive the first-order slow-manifold correction for a broad class of adaptive phase oscillator networks with fast plasticity.
    \item We prove a rigorous sufficient criterion ensuring that the $O(\varepsilon)$ correction
    is genuinely higher-order, and hence that the reduced dynamics are not pairwise.
    \item We verify this criterion explicitly for the adaptive Kuramoto model and identify the
    resulting irreducible triplet terms.
\end{enumerate}

The remainder of the paper is organised as follows. Section~\ref{sec:pairwise} formalises pairwise representability, node-respecting invariance, and the mixed-derivative certificate.
Section~\ref{sec:gspt} recalls the required elements of GSPT,
including the graph expansion and the attraction towards the slow manifold. Section~\ref{sec:model}
introduces the adaptive network class, computes the reduced flow, and proves the emergence of
irreducible triplet interactions. We close with remarks on dynamical significance.

\section{Pairwise and higher-order vector fields}
\label{sec:pairwise}

Since the remainder of the paper concerns phase oscillators, we work directly on $\mathbb{T}^N$.
Let $F=(F_1,\dots,F_N):\mathbb{T}^N\to\mathbb{R}^N$ be a $C^2$ vector field.

\begin{definition}[Pairwise vector field]
\label{def:pairwise}
We say that $F$ is pairwise if, for each $i\in\{1,\dots,N\}$, there exist functions
$G_{ij}\in C^2(\mathbb{T}^2,\mathbb{R})$, $j=1,\dots,N$, such that
\begin{equation}
\label{eq:pairwise_form}
F_i(\theta)=\sum_{j=1}^N G_{ij}(\theta_i,\theta_j).
\end{equation}
We denote the class of such vector fields by $\mathcal{V}_{\mathrm{pair}}$.
\end{definition}

This definition includes intrinsic terms depending only on $\theta_i$, since these can be absorbed
into the $j=i$ summand. In particular, uncoupled dynamics $\dot\theta_i=f_i(\theta_i)$ and constant
drifts are contained in $\mathcal{V}_{\mathrm{pair}}$. No symmetry or homogeneity assumptions are
imposed, and the representation \eqref{eq:pairwise_form} need not be unique.

\begin{definition}[Genuinely higher-order vector field]
\label{def:ho}
We say that $F$ is genuinely higher-order if
\[
F\notin \mathcal{V}_{\mathrm{pair}}.
\]
\end{definition}

The relevant notion of coordinate change is restricted by node semantics.

\begin{definition}[Node-respecting diffeomorphism]
\label{def:node_respecting}
A diffeomorphism $\Phi:\mathbb{T}^N\to\mathbb{T}^N$ is node-respecting if it has the form
\[
\Phi(\theta_1,\dots,\theta_N)
=
\bigl(\phi_1(\theta_{\sigma(1)}),\dots,\phi_N(\theta_{\sigma(N)})\bigr),
\]
where $\sigma\in S_N$ and each $\phi_i:\mathbb{T}\to\mathbb{T}$ is a $C^3$ diffeomorphism.
\end{definition}

\begin{proposition}[Invariance of the pairwise class]
\label{prop:pairwise_invariant}
If $F\in \mathcal{V}_{\mathrm{pair}}$ and $\Phi$ is node-respecting, then the pushforward vector
field $\Phi_*F$ also belongs to $\mathcal{V}_{\mathrm{pair}}$.
\end{proposition}

\begin{proof}
It suffices to consider the case
\[
y_i=\phi_i(\theta_i), \qquad i=1,\dots,N,
\]
since permutation of coordinates clearly preserves pairwise form. Suppose
\[
\dot\theta_i = \sum_{j=1}^N G_{ij}(\theta_i,\theta_j).
\]
Then
\[
\dot y_i
= \phi_i'(\theta_i)\dot\theta_i
= \sum_{j=1}^N \phi_i'(\theta_i)G_{ij}(\theta_i,\theta_j).
\]
Writing $\theta_i=\phi_i^{-1}(y_i)$ and $\theta_j=\phi_j^{-1}(y_j)$, define
\[
\widetilde G_{ij}(u,v)
:=
(\phi_i'\circ \phi_i^{-1})(u)
G_{ij}(\phi_i^{-1}(u),\phi_j^{-1}(v)).
\]
Since $\phi_i\in C^3$, its inverse $\phi_i^{-1}$ is also $C^3$, so the factor
$\phi_i'\circ\phi_i^{-1}$ is $C^2$ (composition of a $C^2$ function with a $C^3$ map).
The second factor $G_{ij}(\phi_i^{-1}(\cdot),\phi_j^{-1}(\cdot))$ is $C^2$ by the chain rule,
since $G_{ij}\in C^2$ and $\phi_i^{-1},\phi_j^{-1}\in C^3$.
Hence $\widetilde G_{ij}\in C^2(\mathbb{T}^2,\mathbb{R})$.
Then
\[
\dot y_i = \sum_{j=1}^N \widetilde G_{ij}(y_i,y_j),
\]
so the transformed vector field is again pairwise.
\end{proof}

\begin{remark}[Coordinate artefacts versus emergence]
Allowing arbitrary diffeomorphisms that mix node coordinates can transform a pairwise vector
field into one with multi-index terms, and vice versa; this is a coordinate artefact. In this paper
we regard two descriptions as equivalent only when they are related by a node-respecting change of variables from Definition~\ref{def:node_respecting}.
Hence higher-order structure in our results is not produced by ``smart'' coordinate mixing, but by
an analytic reduction mechanism (here: Fenichel slow-manifold reduction). Of course, there are more elaborate frameworks to algebraically define higher-order interactions~\cite{AguiarBickDias}, e.g., techniques based upon coupled cell networks~\cite{GolubitskyStewartToeroek} but for our purpose a simpler self-contained approach suffices to illustrate our main point. 
\end{remark}

A convenient sufficient certificate for nonpairwise structure is provided by mixed second derivatives
with respect to two distinct other nodes.

\begin{proposition}[Mixed-derivative vanishing for pairwise fields]
\label{prop:pairwise_mixed}
If $F\in \mathcal{V}_{\mathrm{pair}}$, then for all distinct indices $i,j,k$ one has
\begin{equation}
\label{eq:mixed_zero}
\partial_{\theta_j}\partial_{\theta_k}F_i \equiv 0.
\end{equation}
\end{proposition}

\begin{proof}
By \eqref{eq:pairwise_form},
\[
F_i(\theta)=\sum_{\ell=1}^N G_{i\ell}(\theta_i,\theta_\ell).
\]
If $j\neq k$, then each summand depends on at most one of $\theta_j$ and $\theta_k$, so
\(\partial_{\theta_j}\partial_{\theta_k}\) annihilates every term.
\end{proof}

\begin{corollary}[Mixed-derivative certificate]
\label{cor:irreducibility}
If there exist distinct indices $i,j,k$ and a point $\theta^\ast\in\mathbb{T}^N$ such that
\[
\partial_{\theta_j}\partial_{\theta_k}F_i(\theta^\ast)\neq 0,
\]
then $F\notin \mathcal{V}_{\mathrm{pair}}$.
\end{corollary}

This criterion is used below to rule out pairwise representability in the original node coordinates;
by Proposition~\ref{prop:pairwise_invariant}, the conclusion is unchanged under node-respecting
reparametrisations.

\section{Geometric Singular Perturbation Theory}
\label{sec:gspt}

We recall the fast--slow framework and the graph expansion used later; see
\cite{Fenichel1979,Kuehn2015} for background. Consider a fast--slow system in slow time $\tau$,
\begin{align}
    \varepsilon \dot{x} &= f(x,y,\varepsilon), \label{eq:fast_sys}\\
    \dot{y} &= g(x,y,\varepsilon), \label{eq:slow_sys}
\end{align}
with $x\in\mathbb{R}^m$ fast, $y\in\mathbb{R}^n$ slow, and $0<\varepsilon\ll 1$. Passing to fast time
$s=\tau/\varepsilon$ gives
\[
x' = f(x,y,\varepsilon), \qquad y' = \varepsilon g(x,y,\varepsilon).
\]

\subsection{Limiting problems and normal hyperbolicity}

Setting $\varepsilon=0$ in \eqref{eq:fast_sys}--\eqref{eq:slow_sys} yields the reduced problem
\begin{equation}
\label{eq:reduced}
    0=f(x,y,0), \qquad \dot{y}=g(x,y,0),
\end{equation}
and the layer problem
\begin{equation}
\label{eq:layer}
    x'=f(x,y,0), \qquad y'=0.
\end{equation}
The equilibrium set of the layer problem,
\begin{equation}
\label{eq:C0}
    C_0=\{(x,y): f(x,y,0)=0\},
\end{equation}
is the critical manifold. A compact submanifold $M_0\subset C_0$ is normally hyperbolic if the
spectrum of $\textnormal{D}_x f(x,y,0)$ along $M_0$ is bounded away from the imaginary axis.

\begin{theorem}[Fenichel {\cite{Fenichel1979,Kuehn2015}}]
\label{thm:fenichel}
{Assume that $f$ and $g$ are $C^{r}$ for some finite $r\geq 1$, and let
$M_0\subset C_0$ be a compact normally hyperbolic submanifold. Then for all sufficiently small
$\varepsilon>0$ there exists a locally invariant $C^r$ manifold $M_\varepsilon$ that is
$O(\varepsilon)$-close to $M_0$ and diffeomorphic to $M_0$. After identifying $M_\varepsilon$
with $M_0$ via a Fenichel parametrisation, the reduced vector field on $M_\varepsilon$
converges to the reduced vector field on $M_0$ in $C^r$ as $\varepsilon\to 0$.}
\end{theorem}

\begin{remark}[Stable fibres and the initial fast transient]
\label{rem:stable_fibres}
Fenichel theory gives more than persistence of $M_\varepsilon$. Near a normally hyperbolic attracting slow
manifold, points can be grouped into local stable sets (often called stable fibres): each such set
consists of initial conditions whose trajectories rapidly converge to the same trajectory on $M_\varepsilon$.

Concretely, for all sufficiently small $\varepsilon$ there exist constants $C,c>0$ and a neighbourhood $U$ of
$M_\varepsilon$ such that every $z_0\in U$ is associated with some $z_0^\ast\in M_\varepsilon$, and the corresponding
trajectories satisfy
\begin{equation}
\label{eq:stable_estimate}
\operatorname{dist}\bigl(z(\tau),z^\ast(\tau)\bigr)
\le C e^{-c\tau/\varepsilon}\operatorname{dist}(z_0,z_0^\ast)
\end{equation}
in slow time, for as long as both trajectories remain in $U$.

Hence there is an initial fast transient of length $O(\varepsilon)$ in slow time (equivalently $O(1)$ in fast
time), after which trajectories are exponentially close to the slow manifold. More precisely, entering a fixed
neighbourhood of $M_\varepsilon$ takes slow time $O(\varepsilon)$, while entering an $O(\varepsilon)$ neighbourhood
takes slow time $O(\varepsilon |\log \varepsilon|)$. Therefore the reduced dynamics on $M_\varepsilon$ describe the
long-time motion of an open set of initial conditions, not only trajectories initialised exactly on the
manifold.
\end{remark}

\subsection{Graph case and first-order expansion}

In our application the slow manifold is a graph over the slow variables. {Let $Y$ be a compact smooth manifold and suppose
$M_0$ is given as the graph $x=h_0(y)$ over $Y$, for a $C^3$ map $h_0:Y\to\mathbb{R}^m$, with $M_0$
normally hyperbolic. Working in local coordinates on $Y$, assume in addition that $h_0$ and the full
fast--slow vector field admit $C^3$ extensions to neighbourhoods in the corresponding coordinate charts.
Then for $\varepsilon$ small, $M_\varepsilon$ is also a graph $x=h_\varepsilon(y)$, and the Fenichel graph
admits the first-order expansion}
\begin{equation}
\label{eq:expansion}
    h_\varepsilon(y)=h_0(y)+\varepsilon h_1(y)+o(\varepsilon)
\end{equation}
{in $C^2(Y)$ as $\varepsilon\to 0$. In the later mixed-derivative argument we use exactly this
first-order $C^2$ expansion.}

The coefficients are determined by the invariance equation. If $x=h_\varepsilon(y)$ parametrises
$M_\varepsilon$, then along trajectories on $M_\varepsilon$ one has
\(\dot x = \textnormal{D}_y h_\varepsilon(y)\dot y\), and therefore
\begin{equation}
\label{eq:invariance}
    f(h_\varepsilon(y),y,\varepsilon)
    = \varepsilon \textnormal{D}_y h_\varepsilon(y) g(h_\varepsilon(y),y,\varepsilon).
\end{equation}

\begin{proposition}[First-order correction in the graph case]
\label{prop:h1}
Assume $f,g$ are $C^2$ and that $f$ has no explicit $\varepsilon$-dependence. Then the coefficient
$h_1$ in \eqref{eq:expansion} is the unique solution of
\begin{equation}
\label{eq:h1_linear}
\textnormal{D}_x f(h_0(y),y) h_1(y) = \textnormal{D}_y h_0(y) g(h_0(y),y,0),
\end{equation}
and hence
\begin{equation}
\label{eq:h1}
h_1(y)=\bigl[\textnormal{D}_x f(h_0(y),y)\bigr]^{-1} \textnormal{D}_y h_0(y) g(h_0(y),y,0).
\end{equation}
\end{proposition}

\begin{proof}
Insert \eqref{eq:expansion} into \eqref{eq:invariance}. The $O(1)$ term is
$f(h_0(y),y)=0$, so $M_0\subset C_0$. Subtracting this identity from \eqref{eq:invariance},
dividing by \(\varepsilon\), and letting \(\varepsilon\to 0\) yields
\[
\textnormal{D}_x f(h_0(y),y) h_1(y) = \textnormal{D}_y h_0(y) g(h_0(y),y,0),
\]
because the right-hand side of \eqref{eq:invariance} carries an explicit factor $\varepsilon$ and the
remainder in \eqref{eq:expansion} is $o(\varepsilon)$ in \(C^2(Y)\).
Normal hyperbolicity implies that $\textnormal{D}_x f(h_0(y),y)$ is invertible on $Y$, and this gives
\eqref{eq:h1}.
\end{proof}

\begin{remark}
If $f$ depends explicitly on $\varepsilon$, then an additional term
$\partial_\varepsilon f(h_0(y),y,0)$ appears on the left-hand side of \eqref{eq:h1_linear}.
We restrict to the case without explicit $\varepsilon$-dependence, which covers the adaptive network
models considered below.
\end{remark}

\section{Adaptive network models and emergent triplet interactions}
\label{sec:model}

We now apply the fast--slow framework of Section~\ref{sec:gspt} to adaptive phase oscillator
networks with rapidly evolving coupling weights. The original microscopic equations are strictly
pairwise. The point of the section is to show that elimination of the fast variables can nevertheless
produce a reduced vector field which is not pairwise-representable.
To match the notation in Section~\ref{sec:gspt}, we identify
\(x=A\), \(y=\theta\), \(f_{\mathrm{fast}}:=g\), and \(g_{\mathrm{slow}}:=f\).

\subsection{Model class}

\begin{definition}[Adaptive network model]
\label{def:model}
Let $0<\varepsilon\ll 1$, let $\theta=(\theta_1,\dots,\theta_N)\in\mathbb{T}^N$ be oscillator phases and let
$A=(a_{ij})\in\mathbb{R}^{N\times N}$ be a matrix of coupling weights. On
$\mathcal{X}=\mathbb{T}^N\times\mathbb{R}^{N\times N}$ we consider
\begin{align}
    \dot{\theta} &= f(\theta,A), \label{eq:gen_slow}\\
    \varepsilon \dot{A} &= g(\theta,A), \label{eq:gen_fast}
\end{align}
with components
\begin{align}
    f_i(\theta,A)
    &= \omega_i + \frac{1}{N}\sum_{j=1}^N a_{ij}\,\Gamma(\theta_j-\theta_i),
    \label{eq:f_comp}\\
    g_{ij}(\theta,A)
    &= -a_{ij}+H(\theta_i,\theta_j).
    \label{eq:g_comp}
\end{align}
Here $\omega_i\in\mathbb{R}$ are natural frequencies, and
$\Gamma\in C^2(\mathbb{T},\mathbb{R})$ and $H\in C^2(\mathbb{T}^2,\mathbb{R})$ are
$2\pi$-periodic in each argument.
\end{definition}

\begin{remark}[Adaptive Kuramoto]
\label{rem:kuramoto}
The classical adaptive Kuramoto model corresponds to
\[
\Gamma(\phi)=\sin \phi, \qquad H(u,v)=\alpha+\cos(u-v).
\]
Yet, it is important to note already that our approach applies to much broader classes of adaptive network dynamical systems.
\end{remark}
\subsection{Critical manifold and slow manifold}

Setting $\varepsilon=0$ in \eqref{eq:gen_fast} yields $a_{ij}=H(\theta_i,\theta_j)$, hence the
critical manifold is the graph
\begin{equation}
\label{eq:M0}
M_0 = \{(\theta,A): A=h_0(\theta)\},
\qquad
h_{0,ij}(\theta)=H(\theta_i,\theta_j).
\end{equation}

\begin{proposition}[Global normal hyperbolicity]
\label{prop:hyperbolic}
{$M_0$ is compact and normally hyperbolic. In particular, for all sufficiently small
$\varepsilon>0$ there exists a locally invariant slow manifold}
\[
M_\varepsilon = \{(\theta,A):A=h_\varepsilon(\theta)\}.
\]
{If $\Gamma\in C^3(\mathbb{T})$ and $H\in C^3(\mathbb{T}^2)$, then}
\[
h_\varepsilon(\theta)=h_0(\theta)+\varepsilon h_1(\theta)+o(\varepsilon)
\]
{in the $C^2$ topology.}
\end{proposition}

\begin{proof}
{Since $\mathbb{T}^N$ is compact and $H$ is continuous, $M_0$ is compact.} The layer problem for $A$
with $\theta$ frozen is
\[
a_{ij}'=-a_{ij}+H(\theta_i,\theta_j),
\]
whose linearisation in $A$ is {$D_A g=-I_{N^2}$, so all normal eigenvalues are exactly $-1$. Hence the
normal spectrum is uniformly separated from the imaginary axis and has strictly negative real part; therefore
$M_0$ is uniformly normally hyperbolic and attracting.}
{Fenichel theory applies. Here the slow base is $Y=\mathbb{T}^N$, which is a compact smooth
manifold, so the graph-case discussion from Section~\ref{sec:gspt} applies in local phase charts and gives
the stated first-order expansion in the $C^2$ topology.}
\end{proof}

\begin{remark}[Dynamical relevance of $M_\varepsilon$]
\label{rem:manifold_relevance}
Because the fast linearisation is exactly $-I_{N^2}$, the attraction towards $M_\varepsilon$ is not
marginal but uniformly exponential in fast time. Combining Proposition~\ref{prop:hyperbolic} with
Remark~\ref{rem:stable_fibres}, trajectories starting in a neighbourhood of $M_\varepsilon$ undergo an
initial layer of duration $O(\varepsilon)$ in slow time, after which the distance to a trajectory on the
slow manifold is exponentially small in $\tau/\varepsilon$. Thus the reduced flow on $M_\varepsilon$
describes the physically relevant post-transient behaviour for an open set of initial data. The analysis
is therefore not confined to a negligible region of phase space: the manifold is attracting, and the
reduction captures the long-time dynamics after the fast adaptive variables have relaxed.
\end{remark}

\subsection{First-order correction and reduced phase dynamics}

To compute the first-order correction \(h_1\), we use the invariance equation for the graph
\(A=h_\varepsilon(\theta)\):
\begin{equation}
\label{eq:inv_gen}
g(\theta,h_\varepsilon(\theta))
=
\varepsilon \textnormal{D}_\theta h_\varepsilon(\theta)\,f(\theta,h_\varepsilon(\theta)).
\end{equation}
For \eqref{eq:g_comp}, this becomes
\[
-h_\varepsilon(\theta)+h_0(\theta)
=
\varepsilon \textnormal{D}_\theta h_\varepsilon(\theta)\,f(\theta,h_\varepsilon(\theta)).
\]
Expanding
\[
h_\varepsilon(\theta)=h_0(\theta)+\varepsilon h_1(\theta)+o(\varepsilon)
\]
and matching powers of \(\varepsilon\) yields
\begin{equation}
\label{eq:h1_gen}
h_1(\theta)
=
- \textnormal{D}_\theta h_0(\theta)\,f(\theta,h_0(\theta)).
\end{equation}
Since \(h_{0,ij}(\theta)=H(\theta_i,\theta_j)\) depends only on \((\theta_i,\theta_j)\), this gives
\begin{equation}
\label{eq:h1_comp}
h_{1,ij}(\theta)
=
-\partial_{\theta_i}H(\theta_i,\theta_j)\,f_i(\theta,h_0(\theta))
-\partial_{\theta_j}H(\theta_i,\theta_j)\,f_j(\theta,h_0(\theta)).
\end{equation}

Substituting \(A=h_\varepsilon(\theta)\) into \eqref{eq:f_comp} yields the reduced phase dynamics on
\(M_\varepsilon\):
\begin{equation}
\label{eq:reduced_exact}
\dot{\theta}_i
=
\omega_i+\frac{1}{N}\sum_{j=1}^N h_{\varepsilon,ij}(\theta)\,\Gamma(\theta_j-\theta_i)
=:F_i^\varepsilon(\theta).
\end{equation}
We define the reduced vector field by
\[
F^\varepsilon(\theta):=(F_1^\varepsilon(\theta),\dots,F_N^\varepsilon(\theta)).
\]
Using
\[
h_\varepsilon(\theta)=h_0(\theta)+\varepsilon h_1(\theta)+o(\varepsilon),
\qquad
h_{0,ij}(\theta)=H(\theta_i,\theta_j),
\]
together with \eqref{eq:h1_comp}, we obtain
\begin{align}
\label{eq:reduced_expanded}
\dot{\theta}_i
&=
\omega_i
+\frac{1}{N}\sum_{j=1}^N H(\theta_i,\theta_j)\,\Gamma(\theta_j-\theta_i)
\nonumber\\
&\quad
-\frac{\varepsilon}{N}\sum_{j=1}^N \Gamma(\theta_j-\theta_i)
\Bigl[
\partial_{\theta_i}H(\theta_i,\theta_j)\,f_i(\theta,h_0(\theta))
+
\partial_{\theta_j}H(\theta_i,\theta_j)\,f_j(\theta,h_0(\theta))
\Bigr]
+o(\varepsilon).
\end{align}

Now recall that
\[
f_i(\theta,h_0(\theta))
=
\omega_i+\frac{1}{N}\sum_{k=1}^N H(\theta_i,\theta_k)\,\Gamma(\theta_k-\theta_i).
\]
Substituting this into \eqref{eq:reduced_expanded} and regrouping gives

\begin{align}
\label{eq:reduced_split}
\dot{\theta}_i
&=
\omega_i+\frac{1}{N}\sum_{j=1}^N H(\theta_i,\theta_j)\Gamma(\theta_j-\theta_i)
\nonumber\\
&\quad
+\varepsilon\Bigg[
-\frac{1}{N}\sum_{j=1}^N \Gamma(\theta_j-\theta_i)
\Big(
\partial_{\theta_i}H(\theta_i,\theta_j)\,\omega_i
+
\partial_{\theta_j}H(\theta_i,\theta_j)\,\omega_j
\Big)
\Bigg]
\nonumber\\
&\quad
+\frac{\varepsilon}{N^2}\sum_{j=1}^N\sum_{k=1}^N
\Gamma(\theta_j-\theta_i)
\Big[
-\partial_{\theta_i}H(\theta_i,\theta_j)\,H(\theta_i,\theta_k)\Gamma(\theta_k-\theta_i)
\nonumber\\
&\qquad\qquad\qquad\qquad
-\partial_{\theta_j}H(\theta_i,\theta_j)\,H(\theta_j,\theta_k)\Gamma(\theta_k-\theta_j)
\Big]
\nonumber\\
&\quad
+o(\varepsilon).
\end{align}

\noindent One sees that the first \(O(\varepsilon)\) term is still pairwise, whereas the double-sum term depends on triples
\((\theta_i,\theta_j,\theta_k)\) and is therefore the explicit source of the emergent triplet structure. It is therefore convenient to rewrite \eqref{eq:reduced_split} in the form
\begin{align}
\label{eq:reduced_triplet}
\dot{\theta}_i
&=
\omega_i+\frac{1}{N}\sum_{j=1}^N H(\theta_i,\theta_j)\Gamma(\theta_j-\theta_i)
+\frac{\varepsilon}{N}\sum_{j=1}^N \mathcal{P}_{ij}(\theta)
+\frac{\varepsilon}{N^2}\sum_{j=1}^N\sum_{k=1}^N T_{ijk}(\theta)
+o(\varepsilon),
\end{align}
where
\begin{align}
\label{eq:Pi}
\mathcal{P}_{ij}(\theta)
&=
-\Gamma(\theta_j-\theta_i)\Big(
\partial_{\theta_i}H(\theta_i,\theta_j)\,\omega_i
+
\partial_{\theta_j}H(\theta_i,\theta_j)\,\omega_j
\Big),
\\[0.5em]
\label{eq:Tijk}
T_{ijk}(\theta)
&=
-\Gamma(\theta_j-\theta_i)\,\partial_{\theta_i}H(\theta_i,\theta_j)\,
H(\theta_i,\theta_k)\Gamma(\theta_k-\theta_i)
\nonumber\\
&\quad
-\Gamma(\theta_j-\theta_i)\,\partial_{\theta_j}H(\theta_i,\theta_j)\,
H(\theta_j,\theta_k)\Gamma(\theta_k-\theta_j).
\end{align}

\subsection{Genuinely nonpairwise structure in the reduced flow}

The decomposition \eqref{eq:reduced_triplet} makes the triplet dependence explicit, but the mere appearance
of a double sum does not by itself prove that the reduced vector field is genuinely higher-order: a pairwise
field could be written with redundant summations. What matters is whether the reduced phase dynamics remain
pairwise-representable in the sense of Definition~\ref{def:pairwise}. We now give a rigorous sufficient
criterion, based on the mixed-derivative certificate from Corollary~\ref{cor:irreducibility}.

\begin{theorem}[Emergence of genuinely higher-order reduced dynamics]
\label{thm:main_higher_order}
Let $N\ge 3$, and consider the adaptive network model from Definition~\ref{def:model},
{with $H\in C^3(\mathbb T^2)$ and $\Gamma\in C^3(\mathbb T)$.} Let 
\[
M_\varepsilon=\{(\theta,A):A=h_\varepsilon(\theta)\}
\]
be the associated locally invariant, normally hyperbolic, attracting slow manifold from
Proposition~\ref{prop:hyperbolic}, and let $F^\varepsilon$ denote the reduced phase vector field on
$M_\varepsilon$, defined by \eqref{eq:reduced_exact} and expanded in \eqref{eq:reduced_triplet}.

Suppose there exist distinct indices $i,j,k$ and a point $\theta^*\in\mathbb T^N$ such that
\[
\partial_{\theta_j}\partial_{\theta_k}
\left(
\sum_{r=1}^N\sum_{s=1}^N T_{irs}(\theta)
\right)\Big|_{\theta=\theta^*}\neq 0,
\]
where $T_{irs}$ is defined by~\eqref{eq:Tijk}.
Then the reduced vector field $F^\varepsilon$ is not pairwise representable, i.e.
\[
F^\varepsilon\notin \mathcal{V}_{\mathrm{pair}}
\]
for all sufficiently small $\varepsilon>0$.
Equivalently, the reduced dynamics on $M_\varepsilon$ are genuinely higher-order in the sense of
Definition~\ref{def:ho}.
\end{theorem}

\begin{proof}
By Corollary~\ref{cor:irreducibility}, it suffices to show that there exist distinct indices \(i,j,k\) and a point
\(\theta^\ast\in\mathbb{T}^N\) such that
\[
\partial_{\theta_j}\partial_{\theta_k}\dot{\theta}_i(\theta^\ast)\neq 0.
\]
{Because $H,\Gamma\in C^3$, the expansion \eqref{eq:reduced_triplet} is valid in $C^2$ by
Proposition~\ref{prop:hyperbolic} and the graph expansion from Section~\ref{sec:gspt}. Applying
\(\partial_{\theta_j}\partial_{\theta_k}\) to \eqref{eq:reduced_triplet}, with \(i,j,k\) distinct, is therefore legitimate.
We examine the}
terms on the right-hand side one by one. 
\noindent The intrinsic term \(\omega_i\) is constant, hence annihilated by
\(\partial_{\theta_j}\partial_{\theta_k}\).  The leading-order interaction term
\[
\frac{1}{N}\sum_{\ell=1}^N H(\theta_i,\theta_\ell)\Gamma(\theta_\ell-\theta_i)
\]
is pairwise, since each summand depends only on \((\theta_i,\theta_\ell)\). Therefore
\[
\partial_{\theta_j}\partial_{\theta_k}
\left(
\frac{1}{N}\sum_{\ell=1}^N H(\theta_i,\theta_\ell)\Gamma(\theta_\ell-\theta_i)
\right)
\equiv 0.
\]

\noindent Likewise, each \(\mathcal{P}_{i\ell}\) from \eqref{eq:Pi} is pairwise, since it depends only on
\((\theta_i,\theta_\ell)\). Hence
\[
\partial_{\theta_j}\partial_{\theta_k}
\left(
\frac{1}{N}\sum_{\ell=1}^N \mathcal{P}_{i\ell}(\theta)
\right)
\equiv 0.
\]

\noindent Thus the only possible contribution comes from the triplet term, and \eqref{eq:reduced_triplet} yields
\[
\partial_{\theta_j}\partial_{\theta_k}\dot{\theta}_i
=
\frac{\varepsilon}{N^2}
\partial_{\theta_j}\partial_{\theta_k}
\left(
\sum_{r=1}^N\sum_{s=1}^N T_{irs}(\theta)
\right)
+o(\varepsilon),
\]
By hypothesis, for some point $\theta^* \in \mathbb{T}^N$,
\[
c:=\partial_{\theta_j}\partial_{\theta_k}
\left(
\sum_{r=1}^N\sum_{s=1}^N T_{irs}(\theta)
\right)\Big|_{\theta=\theta^*}\neq 0.
\]
Therefore
\[
\partial_{\theta_j}\partial_{\theta_k}\dot{\theta}_i(\theta^\ast)
=
\frac{\varepsilon}{N^2}c+o(\varepsilon),
\]
which is nonzero for all sufficiently small $\varepsilon>0$. Consequently, by Corollary \ref{cor:irreducibility} the reduced vector field $F^\varepsilon$ does not belong to \(\mathcal{V}_{\mathrm{pair}}\). By
Definition~\ref{def:ho}, the system is therefore genuinely higher-order.
\end{proof}

\begin{corollary}[Adaptive Kuramoto case]
\label{cor:kuramoto}
Let \(N\ge 3\), let
\[
\Gamma(\phi)=\sin\phi,
\qquad
H(u,v)=\alpha+\cos(u-v),
\]
and assume \(\alpha\neq 0\). Then, for all sufficiently small \(\varepsilon>0\), the reduced phase vector
field on \(M_\varepsilon\) is genuinely nonpairwise, that is,
\[
F^\varepsilon\notin\mathcal{V}_{\mathrm{pair}}.
\]
\end{corollary}

\begin{proof}
{For this choice of $\Gamma$ and $H$, both functions are smooth, so the $C^3$ hypothesis of
Theorem~\ref{thm:main_higher_order} is satisfied.}

Fix distinct indices \(i,j,k\). Choose \(\theta^\ast\in\mathbb{T}^N\) by
\[
\theta_i^\ast=0,\qquad
\theta_j^\ast=\frac{\pi}{2},\qquad
\theta_k^\ast=0,
\]
and set \(\theta_\ell^\ast=0\) for all \(\ell\notin\{i,j,k\}\).

For this Kuramoto choice,
\[
\partial_{\theta_i}H(\theta_i,\theta_j)=\sin(\theta_j-\theta_i),
\qquad
\partial_{\theta_j}H(\theta_i,\theta_j)=-\sin(\theta_j-\theta_i),
\]
so \eqref{eq:Tijk} gives
\[
T_{ijk}(\theta)
=
-\sin^2(\theta_j-\theta_i)\,
\bigl(\alpha+\cos(\theta_i-\theta_k)\bigr)\sin(\theta_k-\theta_i)
\]
\[
\qquad
+\sin^2(\theta_j-\theta_i)\,
\bigl(\alpha+\cos(\theta_j-\theta_k)\bigr)\sin(\theta_k-\theta_j).
\]
Define
\[
P(y):=\sin^2 y,
\qquad
Q(x):=(\alpha+\cos x)\sin x.
\]
Then the first line is \(-P(\theta_j-\theta_i)Q(\theta_k-\theta_i)\), so
\[
\partial_{\theta_j}\partial_{\theta_k}\Bigl[-P(\theta_j-\theta_i)Q(\theta_k-\theta_i)\Bigr]
=-P'(\theta_j-\theta_i)Q'(\theta_k-\theta_i).
\]
At \(\theta^\ast\), \(\theta_j^\ast-\theta_i^\ast=\pi/2\), hence
\(P'(\theta_j^\ast-\theta_i^\ast)=2\sin(\pi/2)\cos(\pi/2)=0\), so the first line gives no contribution.
For the second line, with \(x:=\theta_k-\theta_j\),
\[
\partial_{\theta_j}\partial_{\theta_k}\Bigl[P(\theta_j-\theta_i)Q(\theta_k-\theta_j)\Bigr]
=
P'(\theta_j-\theta_i)Q'(x)-P(\theta_j-\theta_i)Q''(x).
\]
Evaluating at \(\theta^\ast\), where \(P'(\pi/2)=0\), \(P(\pi/2)=1\), and \(x=-\pi/2\), gives
\[
\partial_{\theta_j}\partial_{\theta_k}T_{ijk}(\theta^\ast)
=
-\,Q''(-\pi/2)=-\alpha\neq 0.
\]
Now consider
\[
S_i(\theta):=\sum_{r=1}^N\sum_{s=1}^N T_{irs}(\theta).
\]
Since $T_{irs}$ depends only on $(\theta_i,\theta_r,\theta_s)$, the mixed derivative
\(\partial_{\theta_j}\partial_{\theta_k}T_{irs}\) can be nonzero only when
\((r,s)=(j,k)\) or \((r,s)=(k,j)\). Therefore,
\[
\partial_{\theta_j}\partial_{\theta_k}S_i(\theta^\ast)
=
\partial_{\theta_j}\partial_{\theta_k}T_{ijk}(\theta^\ast)
+\partial_{\theta_j}\partial_{\theta_k}T_{ikj}(\theta^\ast).
\]
For \(T_{ikj}\), both terms contain the common factor \(\sin^2(\theta_k-\theta_i)\), which has a
double zero at \(\theta_k=\theta_i\). By the product rule, every term in
\(\partial_{\theta_k}T_{ikj}\) retains a factor of \(\sin(2(\theta_k-\theta_i))\) or
\(\sin^2(\theta_k-\theta_i)\), both of which vanish at \(\theta_k^\ast=\theta_i^\ast\) and are
independent of \(\theta_j\). Hence applying \(\partial_{\theta_j}\) preserves the vanishing prefactors,
giving \(\partial_{\theta_j}\partial_{\theta_k}T_{ikj}(\theta^\ast)=0\), and
\[
\partial_{\theta_j}\partial_{\theta_k}S_i(\theta^\ast)
=
\partial_{\theta_j}\partial_{\theta_k}T_{ijk}(\theta^\ast)
=-\alpha\neq 0.
\]
Therefore the hypothesis of Theorem~\ref{thm:main_higher_order} is satisfied, and the conclusion follows:
\(F^\varepsilon\notin\mathcal{V}_{\mathrm{pair}}\) for all sufficiently small \(\varepsilon>0\).
\end{proof}

\begin{remark}[Why the \(O(\varepsilon)\) triplet term is not negligible]
\label{rem:why_not_negligible}
Although the emergent triplet contribution is multiplied by \(\varepsilon\), it is the leading correction
in the first-order Fenichel reduced dynamics \cite{Fenichel1979,Kuehn2015}. In particular, from
\eqref{eq:reduced_triplet},
\[
F_i^{(0)}(\theta)
:=
\omega_i+\frac{1}{N}\sum_{j=1}^N H(\theta_i,\theta_j)\Gamma(\theta_j-\theta_i),
\]
\[
\dot\theta_i
=
F_i^{(0)}(\theta)
+\frac{\varepsilon}{N}\sum_{j=1}^N \mathcal{P}_{ij}(\theta)
+\frac{\varepsilon}{N^2}\sum_{j=1}^N\sum_{k=1}^N T_{ijk}(\theta)
+o(\varepsilon),
\]
so omitting the triplet term leaves an $O(\varepsilon)$ error rather than an $o(\varepsilon)$ one, and
hence lowers the reduced model from first-order to zeroth-order asymptotic accuracy.

Moreover, over slow-time windows of length \(\tau=O(1/\varepsilon)\), an \(O(\varepsilon)\) perturbation in the
vector field can accumulate to \(O(1)\) phase effects; thus averaged frequencies, phase drifts, and relative
phase offsets may change at leading observable order. Near bifurcation thresholds or partially degenerate
phase configurations (for example, near synchrony), first-order terms often determine
stability selection, so the triplet correction can become the first nonvanishing effective interaction term.

Finally, the relevance is structural, not merely quantitative. Theorem~\ref{thm:main_higher_order} and
Corollary~\ref{cor:kuramoto} show that for sufficiently small \(\varepsilon>0\) the reduced vector field may lie
outside \(\mathcal{V}_{\mathrm{pair}}\). Hence the reduced dynamics need not be simply a small perturbation of a
pairwise model; they can belong to a different interaction class. This aligns with the broader viewpoint that
interaction order can qualitatively alter collective behaviour \cite{Battiston2020PhysRep,Boccalettietal1}.
\end{remark}
\section{Discussion and Outlook}

This work establishes a concrete mechanism for the emergence of effective higher-order interactions from
microscopically pairwise adaptive networks. In the fast-adaptation regime, Fenichel reduction yields an
\(O(\varepsilon)\) correction on the slow manifold that contains irreducible triplet dependence. The mixed-derivative
criterion then shows that the reduced phase vector field need not be pairwise-representable in node coordinates.
In this sense, the class of pairwise network vector fields is not closed under slow-manifold reduction.

For small \(\varepsilon>0\), the conclusion is structural rather than merely perturbative: the reduced dynamics can leave \(\mathcal{V}_{\mathrm{pair}}\) altogether. The adaptive Kuramoto specialisation shows that this is not a formal artefact but a mechanism realised within a standard model class. In the fast-adaptation setting, the source of the higher-order terms is also especially transparent, namely the algebraic substitution of the invariant slow-manifold graph into the phase dynamics. This places the present result alongside other reduction procedures, such as phase reduction and center-manifold/normal-form reduction near bifurcations, that can also generate effective nonpairwise terms.

Since our approach shows that the multiscale reduction framework via GSPT can yield higher-order coupling on a reduced slow manifold, there are several natural conjectures regarding other reduction schemes for network dynamics. Indeed, there are many other routes that explicitly or implicitly use multiscale invariant manifold reduction for network dynamics. Examples are moment closure methods~\cite{KuehnMC} and spectral coarse-graining~\cite{GfellerDeLosRios}. For any class of these reduction methods for networks, we expect the possibility to obtain higher-order interactions in a reduced model. In the reverse direction, it is then evident that higher-order coupling can often also be removed if one can reverse a multiscale reduction and lift/unfold a higher-order model to a higher-dimensional pairwise network model. {Another natural conjectural direction is to investigate how the present fast--slow reduction framework interacts with large-population and continuum limits for adaptive oscillator networks, and more broadly how effective higher-order structure may persist, transform, or re-emerge in macroscopic descriptions.
A further connection worth exploring concerns state-dependent effective pairwise interactions studied through coupling functions with dead zones~\cite{AshwinBickPoignard2019,AshwinBickPoignard2021}. In a formal singular-limit interpretation of the present adaptive model, indicator-type choices of the adaptation law $H$ yield precisely such dead-zone couplings at leading order. Since such choices fall outside the smooth ($C^3$) hypotheses used here, making this connection rigorous would require either a smooth approximation argument or a separate nonsmooth analysis, and this presents an interesting direction for future work.}

\section*{Acknowledgements}
This work was supported by the European Union's Horizon Europe Marie Sk\l odowska-Curie Actions under the ``BeyondTheEdge: Higher-Order Networks and Dynamics'' project (Grant Agreement No.~101120085).

%Related directions include bifurcation and stability effects induced by the emergent triplet term, as well as
%extensions to heterogeneous frequencies, sparse/random interaction structures, and other adaptive coupling laws.
\bibliographystyle{plain}
\bibliography{bibliography}
\end{document}